\documentclass[letterpaper, 10 pt, conference]{ieeeconf}

\bstctlcite{bstctl:etal, bstctl:nodash, bstctl:simpurl}

\usepackage{anyfontsize}
\usepackage{array} 
\usepackage[normalem]{ulem}
\usepackage{stfloats}   %
\usepackage{mathtools}
\usepackage{xcolor}
\usepackage[utf8]{inputenc}
\usepackage{tabularx}
\usepackage{booktabs} %
\usepackage{mathrsfs}
\usepackage{graphics} %

\usepackage{amsmath} %
\usepackage{amssymb}  %
\usepackage{amsthm}
\usepackage{cite}
\usepackage{bm}
\usepackage{acronym}
\usepackage{paralist}
\usepackage{float}
\usepackage[dvipsnames]{xcolor}
\usepackage{epstopdf}
\usepackage{multicol}
\usepackage{tikz}
\usepackage{hyperref}
\usepackage{graphicx}
\usepackage{soul}
\usepackage{macros} %

\usepackage{stfloats}
\fnbelowfloat

 \newcommand{\newtext}{\color{black}}

\makeatletter
\hypersetup{colorlinks=true}
\AtBeginDocument{\@ifpackageloaded{hyperref}
  {\def\@linkcolor{blue}
  \def\@anchorcolor{red}
  \def\@citecolor{red}
  \def\@filecolor{red}
  \def\@urlcolor{black}
  \def\@menucolor{red}
  \def\@pagecolor{red}
\begingroup
  \@makeother\`%
  \@makeother\=%
  \edef\x{%
    \edef\noexpand\x{%
      \endgroup
      \noexpand\toks@{%
        \catcode 96=\noexpand\the\catcode`\noexpand\`\relax
        \catcode 61=\noexpand\the\catcode`\noexpand\=\relax
      }%
    }%
    \noexpand\x
  }%
\x
\@makeother\`
\@makeother\=
}{}}
\makeatother

\newtheorem{Assumption}{Assumption}
\newtheorem{Lemma}{Lemma}

\newtheorem{Remark}{Remark}

\newcommand{\blue}{\color{blue}}

\newcommand{\bequ}{\begin{eqnarray}}
\newcommand{\eequ}{\end{eqnarray}}

\newcommand{\differential}{\mathrm{d}}
\newcommand{\dist}{\mathrm{dist}}

\IEEEoverridecommandlockouts

\def\BibTeX{{\rm B\kern-.05em{\sc i\kern-.025em b}\kern-.08em
    T\kern-.1667em\lower.7ex\hbox{E}\kern-.125emX}}

\definecolor{AA}{RGB}{34,139,34}

\IEEEoverridecommandlockouts
\begin{document}
\bstctlcite{IEEE_b:BSTcontrol}
\title{\LARGE \bf 

On the Equivalence of Stochastic Control and Path Space Formulations for Schrödinger Bridges over Compact Connected Lie Groups

\thanks{This research has been supported by NJIT's startup funds.}
\thanks{H. Mahmood$^*$ and A. Akhtar$^*$ are with the 
Dept. of Mechanical \& Industrial Engineering at the New Jersey Institute of Technology, Newark, NJ 07102, USA (email: \texttt{\{hm576, adeel.akhtar\}@njit.edu}). G.A. Bondar$^{\dag}$ is with the Department of Applied Mathematics at the University of California at Santa Cruz, CA, 95064, USA. (email: \texttt{gbondar@ucsc.edu}) A. Halder$^{\ddag}$ is with the Dept. of Aerospace Engineering at
Iowa State University, Ames, IA 50011, USA (email: \texttt{ahalder@iastate.edu}).}}
\author{Hamza Mahmood$^*$ \and Georgiy A. Bondar$^{\dag}$  \and Abhishek Halder$^{\ddag}$ \and Adeel Akhtar$^*$}
\maketitle

\begin{abstract}

We establish the equivalence between the stochastic optimal control and path space formulations of the Schrödinger bridge problem (SBP) for the kinematic equation on a compact connected Lie group. Using the geometric concepts of horizontal lift and stochastic anti-development, we derive a Girsanov-type change-of-measure result, and show that the expected control energy equals the relative entropy of the controlled path law with respect to the reference Wiener measure. Thus, the SBP is equivalently a path space relative entropy minimization problem subject to prescribed endpoint marginals.
Our result has three useful implications. From an analytic viewpoint, the shown equivalence helps prove the existence and uniqueness of the SB. From a probabilistic viewpoint, it helps interpret the SB as the most probable deviation of the uncontrolled stochastic dynamics consistent with the endpoint constraints. From a computational viewpoint, it allows using static Sinkhorn recursions to directly solve the relative entropy minimization problem and compute the optimal path measure. We illustrate the equivalence numerically on the torus $\mathbb{T}^2$. The code is publicly available at:\\\url{https://github.com/gradslab/LargeDeviationSBP}

\end{abstract}
\section{Introduction}\label{sec:Introduction} 

The Schrödinger bridge problem (SBP) concerns steering a stochastic system between prescribed endpoint probability distributions while minimizing the expected control effort~\cite{chen2021stochastic,CalHal2022}. A substantial systems-control literature has developed around this problem, including theoretical~\cite{CheGeoPav2016,CalHal2020,CalHal_RefSchBri2020,teter2025hopf,teter2025schrodinger} and computational/application-oriented works~\cite{nodozi2023neural,teter2025probabilistic}, primarily for Euclidean state spaces. In a recent work~\cite{mah2026}, we developed a coordinate-free formulation of the SBP on a compact connected Lie group, established existence-uniqueness of solution through the associated Schrödinger system, and obtained a dynamic Sinkhorn scheme for computing the optimal bridge. However, that work did not address the corresponding path space formulation or its large-deviation interpretation, which identifies the SB as the most probable deviation of the uncontrolled stochastic dynamics consistent with the prescribed endpoint constraints.

A related path space formulation for an SBP on Euclidean space was given in~\cite[Theorem 4.2]{teter2025probabilistic} using Girsanov's theorem. For Lie group valued Brownian motion, a Girsanov type result was developed in~\cite{karandikar1983} for matrix Lie groups by transforming a multiplicative Girsanov problem to the Lie algebra through logarithm and exponential mappings. In contrast, the diffusion considered here evolves through tangent spaces that vary along the trajectory. We use horizontal lift and stochastic anti-development~\cite[Ch. 2]{hsu2002} to represent this diffusion in Euclidean coordinates and carry out the required change of measure intrinsically on the Lie group.

We make the following contributions.
\begin{enumerate}
\item We derive a Girsanov type change-of-measure result for the controlled diffusion on the Lie group using horizontal lift and stochastic anti-development (Lemma~\ref{lem:Girsanov_thm_G}).

\item We establish the equivalence between expected control energy and the relative entropy of the controlled path law with respect to the reference Wiener measure (Corollary~\ref{cor:Girsanov_entropy_energy}), thereby recasting the stochastic control version of the SBP as a path space relative entropy minimization problem.

\item We show three implications of this equivalence: analytically, it yields existence-uniqueness of the SB (Theorem~\ref{thm:existence-uniqueness_SBPsolution}); probabilistically, it provides a large-deviation interpretation of the SB as the most probable deviation of the uncontrolled dynamics consistent with the endpoint constraints; numerically, it enables computation of the optimal path measure via static Sinkhorn recursions.
\end{enumerate}

 We numerically illustrate the equivalence on the torus $\mathbb{T}^2$.

\noindent\textbf{Notations.}
For a Lie group $\mathsf G$ and $R\in\mathsf G$, the notation
$\mathsf T_R\mathsf G$ denotes the tangent space of $\mathsf G$ at $R$.
We use $C([0,1];\mathsf G)$ to denote the space of continuous paths from
$[0,1]$ into $\mathsf G$. The symbols $\Delta_{\mathsf G}$ and
$\operatorname{div}$ denote, respectively, the Laplace--Beltrami
operator on $\mathsf G$ and the divergence operator. In a
stochastic differential equation (SDE), the symbol $\circ$ indicates that the
stochastic integral is understood in the Stratonovich sense. 
For two
measures $\mathbb P$ and $\mathbb Q$ on some measure space $\mathcal{X}$, the notation
$\mathbb P\ll\mathbb Q$ means that $\mathbb P$ is absolutely continuous
with respect to $\mathbb Q$. The \emph{relative
entropy} or \emph{Kullback-Leibler (KL) divergence}
\(
D_{\mathrm{KL}}(\mathbb{P}\parallel\mathbb{Q})
:=
\mathbb{E}_{\mathbb{P}}
\left[
\log
\frac{\mathrm{d}\mathbb{P}}
{\mathrm{d}\mathbb{Q}}
\right] = \int_{\mathcal{X}}
\log\left(
\frac{\mathrm{d}\mathbb{P}}
{\mathrm{d}\mathbb{Q}}
\right)
\mathrm{d}\mathbb{P}, \;\text{if }\mathbb{P}\ll\mathbb{Q},
\) and $+\infty$ otherwise, where $\mathbb{E}_{\mathbb{P}}\left[\cdot\right]$ denotes the expectation operator w.r.t. the measure $\mathbb{P}$, and $\frac{\mathrm{d}\mathbb{P}}
{\mathrm{d}\mathbb{Q}}$ denotes the Radon-Nikodym derivative of $\mathbb{P}$ w.r.t. $\mathbb{Q}$. 
The symbol
$\|\cdot\|_2$ denotes the standard Euclidean norm, whereas 
$\int_0^t \langle\alpha_s,\mathrm d\beta_s\rangle:=\int_0^t \sum_{i=1}^n\alpha_s^i\,\mathrm d\beta_s^i$ denotes the Euclidean
inner product in stochastic integral. For a topological space
$A$, its Borel $\sigma$-algebra is denoted by $\mathcal B(A)$. For two $\sigma$-algebras $\mathcal F_1, \mathcal F_2$, their product $\sigma$-algebra
$\mathcal F_1\otimes\mathcal F_2$ is generated by sets $A_1\times A_2$ with
$A_i\in\mathcal F_i$, $i\in\{1,2\}$. A filtration $\{\mathcal F_t\}_{t\in[0,1]}$ on a probability space
$(\Theta,\mathcal F,\mathbb P)$ is an increasing family of sub-$\sigma$-algebras of $\mathcal F$. Other notations will be introduced in situ.

\section{Equivalence of the Stochastic Control and Path Space Formulations}
\label{sec:Problem_formulation}
This section formulates the SBP as a stochastic optimal control problem, develops the path space change-of-measure construction on the Lie group, and establishes the equivalence between the stochastic control and path space relative entropy formulations. 

\subsection{Background}\label{subsec:Background} {Let $\mathsf{G}$ be a compact connected Lie group with identity element $e$, and let $\mathfrak{g}=\mathsf{T}_{e}\mathsf{G}$ be its finite-dimensional Lie algebra of dimension $n$. Consider an inner product $\langle \cdot, \cdot \rangle_{\mathfrak{g}} : \mathfrak{g} \times \mathfrak{g} \to \mathbb{R}$\, and an orthonormal basis $\{e_1^{\mathfrak{g}}, e_2^{\mathfrak{g}}, \dots , e_n^{\mathfrak{g}}\}$ of $\mathfrak{g}$ w.r.t. $\langle \cdot, \cdot \rangle_{\mathfrak{g}}$\,. Let $\hat{\cdot} : \mathbb{R}^{n} \to \mathfrak{g}$, called the ``hat map", denote an isomorphism between the vector spaces $\mathbb{R}^{n}$ and $\mathfrak{g}$. For any \(R\in\mathsf G\), we define the left multiplication \(L_R:\mathsf G\to\mathsf G\) by \(L_R(g)=Rg\). Its differential at the identity,
\(
(\mathsf{D}L_R)_e: \msf{T}_e\mathsf G\to \msf{T}_R\mathsf G,
\)
transports elements of $\mathfrak{g}$ to tangent vectors at \(R\). The \emph{controlled kinematic equation}~\cite[Ch. 5]{bullo2005geometric} on $\mathsf{G}$ is the following (deterministic) differential equation: 
\begin{equation}
\label{eq:kinematic_eq_on_G}
\dot{R}_t
=
(\mathsf{D}L_{R_t})_e
\left(\widehat{\Omega}(R_t,t)\right),\quad\dot{R}_t:=\frac{\differential}{\differential t}R_t,
\end{equation}
where $R_t \in \mathsf{G}$ denotes the configuration (i.e., state) of the system evolving on $\mathsf{G}$ at time $t$, with its time derivative $\dot{R}_t \in \mathsf{T}_{R_{t}}\mathsf{G}$, and the control input ${\Omega} : \mathsf{G} \times [0,1] \rightarrow \mathbb{R}^{n}$.

The \emph{controlled stochastic kinematic equation} is a noisy generalization of \eqref{eq:kinematic_eq_on_G} that describes a diffusion process $R(t)$ on $\mathsf{G}$. This process solves the controlled Stratonovich SDE: 
\begin{equation}
\label{eq:stochastic_kinematic_eq_on_G}
\mathrm{d}R_t
=
(\mathsf{D}L_R)_e\widehat{\Omega}(R,t)\:\mathrm{d}t +
\sigma\sum_{i=1}^{n}
(\mathsf{D}L_R)_e \:e_i^{\mathfrak g}
\circ \mathrm{d}W^i_t.
\end{equation}
In \eqref{eq:stochastic_kinematic_eq_on_G}, $W_t$ is an $\mathbb R^n$-valued standard Brownian motion, and $\sigma > 0$ is the (isotropic) diffusion strength.

\subsection{The Schrödinger Bridge Problem}
Let $\mu$ be the Haar (volume) measure on $\mathsf{G}$ and let $\mathcal{P}_2(\mathsf{G})$ denote the set of PDFs on $\mathsf{G}$ with finite second moments, i.e.,
\begin{align}
\label{eq:P_2_G}
&\mathcal{P}_2(\mathsf{G}) \eqdef \biggl\{ \rho : \mathsf{G} \rightarrow \mathbb{R}_{\geq 0} \bigg|
\int_{\mathsf{G}} \rho \, \mathrm{d}\mu(R) = 1, \notag \\ \;
&\;\;\;\;\;\;\;\;\;\;\;\;\;\;\;\;\;\int_{\mathsf{G}} {\dist(R,R_0)^2} \, \rho \, \mathrm{d}\mu(R) < \infty \biggr\},
\end{align} 
where $\dist(R,R_0)$ denotes the \emph{Riemannian distance} of $R \in \mathsf{G}$ from a fixed $R_0\in \mathsf{G}$. Given $\rho_0, \rho_1 \in \mathcal{P}_2(\mathsf{G})$, let $\mathcal{P}_{01}(\mathsf{G})$ be the set of all probability density trajectories $\rho(\cdot, t)$ in $\mathcal{P}_2(\mathsf{G})$ that are continuous in $t \in [0,1]$ with endpoints $\rho(\cdot,0) = \rho_0$ and $\rho(\cdot,1) = \rho_1$, i.e.,   
\begin{align}
\label{eq:P_01_G}
&\mathcal{P}_{01}(\mathsf{G}) \eqdef \biggl\{ \rho: \mathsf{G} \times [0,1] \to \mathbb{R}_{\geq 0} \,\bigg|\, \rho(\cdot,0) = \rho_0, \rho(\cdot,1) = \rho_1, \notag \\ \;
&\;\;\;\;\;\;\;\;\;\;\;\;\;\;\;\;\;\;\;\;\;\;\,\text{for each }t \in [0,1], \rho(\cdot, t) \in \mathcal{P}_2(\mathsf{G}) \biggr\}.
\end{align}
Also, let $\mathcal{V}$ be the set of all Markovian control policies on $\mathsf{G}$ with finite energy, i.e., 
{\begin{equation}
\label{eq:Controls}
\begin{aligned}
    &\mathcal{V} \eqdef 
    \biggl\{ {\Omega} : \mathsf{G} \times [0,1] \rightarrow \mathbb{R}^{n} \;\bigg|\;
    \Omega \text{ is measurable and } \\[1pt]
    &\;\;\forall \rho \in \mathcal{P}_{01}(\mathsf{G}),  \int_{0}^{1}\int_{\mathsf{G}} \|\Omega(R,t)\|_{2}^2\,\rho(R,t)\, \differential\mu(R)\:\differential t < \infty
    \biggr\}.
\end{aligned}
\end{equation}
}  
{\newtext 
The main idea in SBP is to design a control $\Omega \in \mathcal{V}$ that steers the stochastic state $R$ on $\mathsf{G}$ from a prescribed PDF $\rho_0$ to another $\rho_1$ over a finite time horizon $[0,1]$ with minimum expected control effort. The controlled stochastic process $R$ is governed by the stochastic differential equation (SDE)~\eqref{eq:stochastic_kinematic_eq_on_G} together with 
\(
R_0 \sim \rho_{0}(R), \; R_1 \sim \rho_{1}(R).
\)
Then, the corresponding controlled state PDF $\rho(R,t)$ must satisfy the constraints~\cite[equation (6b)]{mah2026}: 
\begin{subequations}
\label{eq:FPK_PDE}
\begin{align}
&\partial_t \rho = -\mathrm{div}(\rho (\mathsf D L_R)_e\widehat{\Omega}(R,t)) + \frac{\sigma^2}{2} \Delta_{\mathsf{G}} \rho(R,t),
\label{1st_FPK_PDE} \\
&\rho(R,0) = \rho_0(R), \qquad \rho(R,1) = \rho_1(R). \label{2nd_FPK_PDE}
\end{align}
\end{subequations}
}

We can now formally state the SBP for the kinematic equation on $\mathsf{G}$ as follows.
\begin{problem}
\label{problem:SBP_on_G}
Solve stochastic optimal control problem:
\begin{subequations}
\label{eq:SBP_on_G}
\begin{align}
&\inf_{\left(\rho, \Omega\right) \in \mathcal{P}_{01}(\mathsf{G}) \times  \mathcal{V}} \;  \int_{0}^{1} \int_{\mathsf{G}} \frac{1}{2} \|\Omega(R,t)\|_{2}^{2} \, \rho(R,t) \differential\mu(R) \mathrm{d}t  \label{1st_SBP_on_G} \\
&\quad\mathrm{subject\,to\,\,\,\,\,\,\,\,\,{constraints}}~\eqref{eq:FPK_PDE}.
\end{align}
\end{subequations}
\end{problem}

{
}

Problem~\ref{problem:SBP_on_G} therefore seeks a finite-energy Markovian control that steers the stochastic process on $\mathsf G$ between the prescribed endpoint distributions while minimizing the expected control energy. The associated density evolution $\rho \in \mathcal{P}_{01}(\mathsf{G})$ satisfying~\eqref{eq:FPK_PDE} generates a stochastic interpolant between these marginals. We next develop the path space formulation and establish its equivalence with the control formulation \eqref{eq:SBP_on_G}.

\subsection{Path Space Setup and Stochastic Anti-Development}
We recall that a stochastic process $Y=(Y_t)_{t\in[0,1]}$ with values in
a measurable space $(E,\mathcal E)$ is said to be \emph{progressively
measurable}~\cite[Definition 1.11]{shr1991} with respect to a filtration
$\{\mathcal F_t\}_{t\in[0,1]}$ if, for every $t\in[0,1]$, the mapping
\(
(s,\theta)\longmapsto Y_s(\theta),
\;\;(s,\theta)\in[0,t]\times\Theta,
\)
is measurable from
$\bigl([0,t]\times\Theta,
\mathcal B([0,t])\otimes\mathcal F_t\bigr)$
to $(E,\mathcal E)$. When the filtration is the usual augmentation
of the canonical filtration with respect to a probability measure
$\mathbb P$, we simply say that $Y$ is progressively measurable
with respect to the canonical filtration under $\mathbb P$. In
particular, the Assumption~\ref{assump:progressive_measurability} below requires that the process
$t\mapsto\Omega(R_t,t)$ be progressively measurable in this sense. An SDE is said to be \emph{well posed in law} for a given
initial distribution if a weak solution~\cite[Definition 3.1]{shr1991} exists and its probability
law is unique: that is, if
$(R,W)$ and $(\widetilde R,\widetilde W)$ are any two weak solutions,
possibly defined on different filtered probability spaces but having
the same prescribed initial distribution, then $R$ and
$\widetilde R$ induce the same probability measure on the path space
$C([0,1];\mathsf G)$.  

{\newtext Let $X_i$ be the left-invariant vector field on $\mathsf G$
associated with $e_i^{\mathfrak g}$, given by}
\[
X_i(R):=(\mathsf D L_R)_e e_i^{\mathfrak g},
\qquad i=1,\ldots,n.
\]
For $a=(a_1,\ldots,a_n)^\top\in\mathbb R^n$, set
\(
\mathsf E_R a:=\sum_{i=1}^n a_iX_i(R).
\)
This map $\mathsf E_R:\mathbb R^n\to \mathsf{T}_{R}\mathsf G$ is an isometry. Indeed, the tangent space $\mathsf T_R\mathsf G$ is endowed with the inner product induced by the left-invariant Riemannian metric, namely
\[
\langle v,w\rangle_R
=
\left\langle
(\mathsf D L_{R^{-1}})_R v,\,
(\mathsf D L_{R^{-1}})_R w
\right\rangle_{\mathfrak g}.
\]
Since $\{e_i^{\mathfrak g}\}_{i=1}^n$ is orthonormal in $\mathfrak g$, it follows that $\mathsf E_R:\mathbb R^n\to \mathsf T_R\mathsf G$ preserves inner products and hence is an isometry. 
Let
\(
\mathcal X:=C([0,1];\mathsf G),
\)
equipped with its Borel $\sigma$-algebra and canonical filtration, and
let $\mathbb W_{\rho_0}^{\sigma}$ be the Wiener measure on $\mathcal X$ associated with the reference diffusion
\begin{equation}
\label{eq:Girsanov_reference}
\mathrm dR_t
=
\sigma\sum_{i=1}^n
X_i(R_t)\circ\mathrm dW_t^i,
\qquad
R_0\sim\rho_0,
\end{equation}
where $(W_t)_{t\ge0}$ is an $\mathbb R^n$-valued standard Brownian motion and $\rho_0 \in \mathcal{P}_2(\mathsf{G})$ given. 

Let $\nabla$ denote the Levi-Civita connection associated with the
left-invariant Riemannian metric on $\mathsf G$ induced by
$\langle\cdot,\cdot\rangle_{\mathfrak g}$~\cite[Section 21.3]{gallier2020}. Let
$U_t:\mathbb R^n\to\mathsf T_{R_t}\mathsf G$ denote the horizontal lift~\cite[Chapter 2]{hsu2002}
of the canonical process $R_t$ with respect to $\nabla$, initialized by
$U_0=\mathsf E_{R_0}$, and let
\begin{equation}
\label{eq:Girsanov_antidevelopment}
\beta_t
:=
\frac1{\sigma}
\int_0^t U_s^{-1}\circ\mathrm dR_s
\end{equation}
denote its stochastic anti-development. Under
$\mathbb W_{\rho_0}^{\sigma}$, $\beta$ is a standard
$\mathbb R^n$-valued Brownian motion. {\newtext Next we define the map $\mathsf O_t:\mathbb R^n\to\mathbb R^n$ as the
composition of $U_t^{-1}$ and $\mathsf E_{R_t}$:} namely
\(
\mathsf O_t:=U_t^{-1}\mathsf E_{R_t}.
\)
This map $\mathsf O_t$ is an orthogonal transformation of $\mathbb{R}^n$.

For a Markovian control
$\Omega \in \mathcal{V}$, we let
\begin{equation}
\label{eq:b^Omega}
b^\Omega(R,t)
:=
\mathsf E_R\Omega(R,t)
=
(\mathsf D L_R)_e\widehat{\Omega}(R,t),
\end{equation}
and
\begin{equation}
\label{eq:alpha_t}
\alpha_t
:=
\frac1{\sigma}
\mathsf O_t\Omega(R_t,t).
\end{equation}
We now state the assumptions needed for our setup.  
\begin{Assumption}[Progressive measurability]
\label{assump:progressive_measurability}
The process 
\(
t\longmapsto\Omega(R_t,t)
\)
is progressively measurable with respect to the canonical filtration
under $\mathbb W_{\rho_0}^{\sigma}$.
\end{Assumption}

\begin{Assumption}[Well-posedness in law]
\label{assump:wellposedness}
The controlled Stratonovich SDE
\begin{equation}
\label{eq:Girsanov_controlled}
\mathrm dR_t
=
b^\Omega(R_t,t)\,\mathrm dt
+
\sigma\sum_{i=1}^n
X_i(R_t)\circ\mathrm dW_t^i,
\quad
R_0\sim\rho_0,
\end{equation}
is well posed in law.
\end{Assumption}

\begin{Assumption}[Novikov condition]
\label{assump:Novikov}
The control $\Omega \in \mathcal{V}$ satisfies
\begin{equation}
\label{eq:Girsanov_Novikov}
\mathbb E_{\mathbb W_{\rho_0}^{\sigma}}
\left[
\exp\left(
\frac{1}{2\sigma^2}
\int_0^1
\|\Omega(R_t,t)\|_2^2\,\mathrm dt
\right)
\right]
<\infty.
\end{equation}
\end{Assumption}

\begin{Remark}
The Assumptions 1-3 ensure the Girsanov change-of-measure construction on $\mathsf{G}$\,. In particular, progressive measurability guarantees that the control can be used as an admissible integrand in the stochastic exponential, while well-posedness in law ensures that the process obtained after the change of measure can be identified uniquely with the controlled diffusion. The Novikov condition, in turn, guarantees that the corresponding stochastic exponential is a martingale and hence defines a probability measure on the path space. {\newtext These standard assumptions provide natural sufficient conditions for the
change-of-measure argument, and are mild enough to encompass a broad class of controlled diffusions relevant to our setting.} 
\end{Remark}

\subsection{Girsanov Type Change of Measure}
Under
Assumptions~\ref{assump:progressive_measurability}--
\ref{assump:Novikov}, we can relate the controlled and uncontrolled path measures through an explicit Radon-Nikodym derivative, which will be essential for the relative entropy formulation of the SBP. Compared with the Euclidean case, the main additional difficulty on
$\mathsf G$ is that the driving noise takes values in \emph{varying tangent
spaces} rather than in a fixed vector space. To apply the classical Girsanov theorem~\cite[Theorem 8.6.6]{bernt2013}, we will therefore use stochastic anti-development to
represent the diffusion in Euclidean coordinates, and relate the
horizontal frame to the left-invariant frame through the orthogonal map
$\mathsf O_t$. This idea is used in the following lemma.

\begin{Lemma}[Girsanov type result for a controlled diffusion on $\mathsf G$]
\label{lem:Girsanov_thm_G}
Under
Assumptions~\ref{assump:progressive_measurability}--
\ref{assump:Novikov}, 
the stochastic exponential
\begin{equation}
\label{eq:Girsanov_Z}
Z_t
:=
\exp\left(
\int_0^t
\langle\alpha_s,\mathrm d\beta_s\rangle
-
\frac12
\int_0^t
\|\alpha_s\|_2^2\,\mathrm ds
\right)
\end{equation}
is a strictly positive
$\mathbb W_{\rho_0}^{\sigma}$-martingale with
$\mathbb E_{\mathbb W_{\rho_0}^{\sigma}}[Z_t]=1$. Define the probability measure $\mathbb P^\Omega$ on $\mathcal X$ by
\[
\frac{\mathrm d\mathbb P^\Omega}
     {\mathrm d\mathbb W_{\rho_0}^{\sigma}}
=
Z_1.
\]
Then
\begin{equation}
\label{eq:Girsanov_shifted_BM}
\beta_t^\Omega
:=
\beta_t-\int_0^t\alpha_s\,\mathrm ds
\end{equation}
is a standard $\mathbb R^n$-valued Brownian motion under
$\mathbb P^\Omega$. Furthermore, under $\mathbb P^\Omega$, the canonical
process $R$ has the law of the controlled diffusion
\eqref{eq:Girsanov_controlled}, and
\begin{align}
\frac{\mathrm d\mathbb P^\Omega}
     {\mathrm d\mathbb W_{\rho_0}^{\sigma}}
&=
\exp\Biggl(
\frac1{\sigma}
\int_0^1
\left\langle
\mathsf O_t\Omega(R_t,t),
\mathrm d\beta_t
\right\rangle
\notag\\
&\qquad\qquad
-
\frac1{2\sigma^2}
\int_0^1
\|\Omega(R_t,t)\|_2^2\,\mathrm dt
\Biggr).
\label{eq:Girsanov_RN}
\end{align}
\end{Lemma}

\begin{proof}
By the Novikov condition (Assumption~\ref{assump:Novikov}), %
\[
Z_t
=
\exp\left(
\int_0^t
\langle\alpha_s,\mathrm d\beta_s\rangle
-
\frac12
\int_0^t
\|\alpha_s\|_2^2\,\mathrm ds
\right)
\]
is a strictly positive martingale satisfying
\(
\mathbb E_{\mathbb W_{\rho_0}^{\sigma}}[Z_t]=1.
\)
We can therefore define the probability measure
$\mathbb P^\Omega$ on $\mathcal X$ by
\(
\mathrm d\mathbb P^\Omega
=
Z_1\,\mathrm d\mathbb W_{\rho_0}^{\sigma}.
\)
Since $Z_1>0$,
$\mathbb W_{\rho_0}^{\sigma}$-almost surely, the two measures are
mutually absolutely continuous: 
\(
\mathbb P^\Omega
\ll
\mathbb W_{\rho_0}^{\sigma} \text{ and } 
\mathbb W_{\rho_0}^{\sigma}
\ll
\mathbb P^\Omega.
\)
By the classical Girsanov theorem~\cite[Theorem 8.6.6]{bernt2013}, 
\begin{equation}
\label{eq:beta_controlled}
\beta_t^\Omega
:=
\beta_t
-
\int_0^t\alpha_s\,\mathrm ds
\end{equation}
is a standard $\mathbb R^n$-valued Brownian motion under
$\mathbb P^\Omega$.
Equivalently, we get
\begin{equation}
\label{eq:d_beta_t}
\mathrm d\beta_t
=
\mathrm d\beta_t^\Omega
+
\alpha_t\,\mathrm dt.
\end{equation}
From the stochastic anti-development in~\eqref{eq:Girsanov_antidevelopment}, 
we have  
\begin{equation}
\label{eq:development_relation_proof}
\mathrm dR_t
=
\sigma U_t\circ\mathrm d\beta_t.
\end{equation}
Substituting the expression for $\mathrm d\beta_t$ from \eqref{eq:d_beta_t} into \eqref{eq:development_relation_proof} yields
\begin{align}
\mathrm dR_t
&=
\sigma U_t
\circ
\left(
\mathrm d\beta_t^\Omega
+
\alpha_t\,\mathrm dt
\right)
=
\sigma U_t\circ\mathrm d\beta_t^\Omega
+
\sigma U_t\alpha_t\,\mathrm dt.
\label{eq:after_girsanov_horizontal}
\end{align}
Since the second term is of finite variation, no Stratonovich
correction arises from the change
$\mathrm d\beta_t\mapsto
\mathrm d\beta_t^\Omega+\alpha_t\,\mathrm dt$.
Using the definition of $\alpha_t$ {\newtext in~\eqref{eq:alpha_t} and of $b^\Omega$ in~\eqref{eq:b^Omega}, we have 
$
\sigma U_t\alpha_t
=
b^\Omega(R_t,t).
$ Plugging this in~\eqref{eq:after_girsanov_horizontal} gives   
}

\begin{equation}
\label{eq:controlled_horizontal}
\mathrm dR_t
=
b^\Omega(R_t,t)\,\mathrm dt
+
\sigma U_t\circ\mathrm d\beta_t^\Omega.
\end{equation}
The infinitesimal generator of
\eqref{eq:controlled_horizontal} is
\(
\mathcal L_t^\Omega
=
b^\Omega(\cdot,t)
+
\frac{\sigma^2}{2}\Delta_{\mathsf G},
\)
where $b^\Omega$ is understood as the corresponding first-order
differential operator. On the other hand, the Stratonovich SDE
\(
\mathrm dR_t
=
b^\Omega(R_t,t)\,\mathrm dt
+
\sigma
\sum_{i=1}^n
X_i(R_t)\circ\mathrm dW_t^{i}
\)
has generator
\begin{align}
b^\Omega(\cdot,t)
+
\frac{\sigma^2}{2}
\sum_{i=1}^n X_i^2
&=
b^\Omega(\cdot,t)
+
\frac{\sigma^2}{2}\Delta_{\mathsf G}.
\end{align}
Thus, the horizontal representation \eqref{eq:controlled_horizontal}
and the left-invariant representation \eqref{eq:Girsanov_controlled}
have the same infinitesimal generator. By the assumed well-posedness
in law of the controlled SDE \eqref{eq:Girsanov_controlled} (Assumption~\ref{assump:wellposedness}), their
path laws coincide. Consequently, $\mathbb P^\Omega$ is precisely the
path law of the controlled diffusion \eqref{eq:Girsanov_controlled}.

Since $Z_t$ is a martingale with respect to the filtration that
contains $R_0$, we also have
\(
\mathbb E_{\mathbb W_{\rho_0}^{\sigma}}
[Z_1\mid R_0]
=
1.
\)
Hence the change of measure leaves the initial distribution
unchanged:
\(
R_0\sim\rho_0\,\mathrm d\mu
\;\;
\text{under }\mathbb P^\Omega.
\)
Since
$
b^\Omega(R_t,t)
=
\mathsf E_{R_t}\Omega(R_t,t),
$
we obtain
\begin{align*}
U_t^{-1}b^\Omega(R_t,t)
&=
U_t^{-1}
\mathsf E_{R_t}\Omega(R_t,t)
\notag =
\mathsf O_t\Omega(R_t,t).
\label{eq:control_horizontal_coordinates}
\end{align*}
Since $\mathsf O_t$ is orthogonal,
\begin{equation}
\label{eq:norm_alpha}
\|\alpha_t\|_2^2
=
\frac{1}{\sigma^2}
\|\Omega(R_t,t)\|_2^2.
\end{equation}
Substituting
\(
\alpha_t
=
\frac{1}{\sigma}
\mathsf O_t\Omega(R_t,t)
\)
into the definition of $Z_1$ and using
\eqref{eq:norm_alpha} give~\eqref{eq:Girsanov_RN}, which completes the proof.
\end{proof}

\subsection{Relative Entropy and Control Energy}
Lemma~\ref{lem:Girsanov_thm_G} helps to establish a link between relative entropy and control energy, formalized in Corollary \ref{cor:Girsanov_entropy_energy}. This will be useful in the development that follows.
\begin{corollary}[Relative entropy and control energy]
\label{cor:Girsanov_entropy_energy}
Under the hypotheses of
Lemma~\ref{lem:Girsanov_thm_G}, suppose additionally that
\begin{equation}
\label{eq:Girsanov_finite_controlled_energy}
\mathbb E_{\mathbb P^\Omega}
\left[
\int_0^1
\|\Omega(R_t,t)\|_2^2\,\mathrm dt
\right]
<\infty.
\end{equation}
Then {\newtext the KL divergence $D_{\mathrm{KL}}
\left(
\mathbb P^\Omega
\parallel
\mathbb W_{\rho_0}^{\sigma}
\right)$ is given by }
\begin{equation}
\label{eq:Girsanov_entropy_energy}
D_{\mathrm{KL}}
\left(
\mathbb P^\Omega
\parallel
\mathbb W_{\rho_0}^{\sigma}
\right)
=
\frac{1}{2\sigma^2}
\mathbb E_{\mathbb P^\Omega}
\left[
\int_0^1
\|\Omega(R_t,t)\|_2^2\,\mathrm dt
\right].
\end{equation}
\end{corollary}
\begin{proof}
\medskip
\noindent
From \eqref{eq:beta_controlled}, we have
\(
\mathrm d\beta_t
=
\mathrm d\beta_t^\Omega
+
\alpha_t\,\mathrm dt.
\)
Consequently, \eqref{eq:Girsanov_RN} gives 
\begin{align}
\log
\frac{\mathrm d\mathbb P^\Omega}
{\mathrm d\mathbb W_{\rho_0}^{\sigma}}
&=
\int_0^1
\langle\alpha_t,\mathrm d\beta_t\rangle
-
\frac12
\int_0^1
\|\alpha_t\|_2^2\,\mathrm dt
\notag\\
&=
\int_0^1
\langle
\alpha_t,\mathrm d\beta_t^\Omega
\rangle
+
\frac12
\int_0^1
\|\alpha_t\|_2^2\,\mathrm dt.
\label{eq:log_RN_under_P}
\end{align}
Under the additional finite energy assumption
\eqref{eq:Girsanov_finite_controlled_energy}, relation \eqref{eq:norm_alpha} implies
\[
\mathbb E_{\mathbb P^\Omega}
\left[
\int_0^1
\|\alpha_t\|_2^2\,\mathrm dt
\right]
<\infty.
\]
Therefore,
\(
\int_0^t
\langle
\alpha_s,\mathrm d\beta_s^\Omega
\rangle
\)
is a square-integrable martingale under $\mathbb P^\Omega$, and hence
\[
\mathbb E_{\mathbb P^\Omega}
\left[
\int_0^1
\langle
\alpha_t,\mathrm d\beta_t^\Omega
\rangle
\right]
=
0.
\]
Taking $\mathbb P^\Omega$-expectations in
\eqref{eq:log_RN_under_P} gives
\begin{align}
D_{\mathrm{KL}}
\left(
\mathbb P^\Omega
\parallel
\mathbb W_{\rho_0}^{\sigma}
\right)
&=
\mathbb E_{\mathbb P^\Omega}
\left[
\log
\frac{\mathrm d\mathbb P^\Omega}
{\mathrm d\mathbb W_{\rho_0}^{\sigma}}
\right]
\notag\\
&=
\frac12
\mathbb E_{\mathbb P^\Omega}
\left[
\int_0^1
\|\alpha_t\|_2^2\,\mathrm dt
\right]
\notag\\
&=
\frac{1}{2\sigma^2}
\mathbb E_{\mathbb P^\Omega}
\left[
\int_0^1
\|\Omega(R_t,t)\|_2^2\,\mathrm dt
\right],
\end{align}
which proves \eqref{eq:Girsanov_entropy_energy}.    
\end{proof}

\subsection{Equivalence of the Two Formulations}
For the equivalence below, we consider admissible controls satisfying Assumptions~\ref{assump:progressive_measurability}--\ref{assump:Novikov}.
With
$
\mathcal{X} = C([0,1];\mathsf G),$ equipped with its Borel $\sigma$-algebra, let $\mathcal M(\mathcal{X})$ denote the collection of all probability measures on $\mathcal X$. Let
$\mathbb W_{\rho_0}^{\sigma}\in\mathcal M(\mathcal{X})$ denote the Wiener measure on
$\mathcal X$ induced by the uncontrolled diffusion
\begin{equation}
\label{eq:UncontrolledBrownian_thm}
\mathrm dR_t
=
\sigma\sum_{i=1}^{n}
(\mathsf D L_{R_t})_e e_i^{\mathfrak g}
\circ\mathrm dW_t^i,
\qquad
R_0\sim\rho_0\,\mathrm d\mu.
\end{equation}

For any admissible pair $\left(\rho, \Omega\right) \in \mathcal{P}_{01}(\mathsf{G}) \times  \mathcal{V}$ in
Problem~\ref{problem:SBP_on_G}, let  $\mathbb P^\Omega\in\mathcal M(\mathcal{X})$ denote the
path law on $\mathcal X$ of the corresponding controlled diffusion~\eqref{eq:stochastic_kinematic_eq_on_G}.
Since the time-$t$ marginal of $\mathbb P^\Omega$ has density
$\rho(\cdot,t)$ with respect to $\mu$, we have
\begin{align}
&\mathbb E_{\mathbb P^\Omega}
\left[
\int_0^1
\|\Omega(R_t,t)\|_2^2\,\mathrm dt
\right] \notag \\
&\qquad\quad=
\int_0^1\int_{\mathsf G}
\|\Omega(R,t)\|_2^2
\rho(R,t)\,
\mathrm d\mu(R)\,\mathrm dt.
\label{eq:energy_marginal_identity}
\end{align}
Since every admissible control $\Omega\in\mathcal V$ has finite
control energy and is assumed to satisfy Assumptions~\ref{assump:progressive_measurability}--
\ref{assump:Novikov}, Corollary~\ref{cor:Girsanov_entropy_energy} applies and yields
\begin{equation}
\label{eq:entropy_energy_thm}
D_{\mathrm{KL}}
\left(
\mathbb P^\Omega
\parallel
\mathbb W_{\rho_0}^{\sigma}
\right)
=
\frac{1}{2\sigma^2}
\int_0^1\int_{\mathsf G}
\|\Omega(R,t)\|_2^2
\rho(R,t)\,
\mathrm d\mu(R)\,\mathrm dt.
\end{equation} 
Next, we define
\begin{align}
\Pi_{01}
&:=
\Bigl\{
\mathbb M\in\mathcal M(\mathcal X)
\,\Big|\,
\mathbb M
\text{ has marginals }
\rho_0\,\mathrm d\mu
\notag\\
&\qquad\qquad
\text{ and }
\rho_1\,\mathrm d\mu
\text{ at }
t=0,1
\Bigr\}.
\label{eq:Pi01LieGroup}
\end{align}
Since every admissible measure in $\Pi_{01}$ possesses the prescribed endpoint marginals,
minimizing
\eqref{1st_SBP_on_G}
is equivalent to minimizing
the relative entropy
\begin{align}
\inf_{\mathbb P\in\Pi_{01}}
D_{\mathrm{KL}}
\left(
\mathbb P
\parallel
\mathbb W_{\rho_0}^{\sigma}
\right).
\label{eq:SBPKL}
\end{align}

\section{Existence-Uniqueness of the SBP Solution}
\label{sec:existence_uniqueness_min_pair}
As an analytic implication of the equivalence established above, we now state and prove the existence-uniqueness result for the SBP. 
\begin{theorem}[Existence-uniqueness of the SBP solution on $\mathsf G$]
\label{thm:existence-uniqueness_SBPsolution}
Let $\mathsf G$ be a compact connected Lie group with Lie algebra
$\mathfrak g$ of dimension $n$, endowed with the left-invariant
Riemannian metric induced by
$\langle\cdot,\cdot\rangle_{\mathfrak g}$, and let $\sigma>0$.
Suppose that the prescribed endpoint densities $\rho_0, \rho_1$ satisfy
\[
\rho_0,\rho_1\in \mathcal{P}_2(\mathsf{G})\, \cap\, C^2(\mathsf G),
\,
\rho_0(R)>0,\, \rho_1(R)>0,
\,
\forall\,R\in\mathsf G.
\]
Assume further that every admissible control
$\Omega\in\mathcal V$ satisfies
Assumptions~\ref{assump:progressive_measurability}--
\ref{assump:Novikov}.
Then Problem~\ref{problem:SBP_on_G} admits a unique minimizing pair
\(
\left(
\rho_\sigma^{\mathrm{opt}},
\Omega_\sigma^{\mathrm{opt}}
\right).
\)
\end{theorem}

\begin{proof}

We first show that this relative entropy minimization problem has finite value. Let $p_t^\sigma(R,S)$ denote the transition density, with respect to
$\mu$, of the uncontrolled diffusion
\eqref{eq:UncontrolledBrownian_thm}. Since $\mathsf G$ is compact and
connected, $p_t^\sigma$ is smooth and
strictly positive for every $t>0$. In particular, for $t=1$, there exist
constants $0<c_\sigma\leq C_\sigma<\infty$ such that
\begin{equation}
\label{eq:heat_kernel_bounds}
c_\sigma
\leq
p_1^\sigma(R,S)
\leq
C_\sigma,
\qquad
\forall\,R,S\in\mathsf G.
\end{equation}
The joint law of $(R_0,R_1)$ under
$\mathbb W_{\rho_0}^{\sigma}$ is therefore
\begin{equation}
\label{eq:reference_endpoint_law}
\mathbb W_{01}^{\sigma}
(\mathrm dR,\mathrm dS)
=
\rho_0(R)\,
p_1^\sigma(R,S)\,
\mathrm d\mu(R)\,\mathrm d\mu(S).
\end{equation}
Consider now the product coupling of the prescribed endpoint
distributions,
\begin{equation}
\label{eq:product_endpoint_coupling}
\pi(\mathrm dR,\mathrm dS)
=
\rho_0(R)\rho_1(S)\,
\mathrm d\mu(R)\,\mathrm d\mu(S).
\end{equation}
Since $p_1^\sigma(R,S)>0$, we have
$\pi\ll\mathbb W_{01}^{\sigma}$, with
\begin{equation}
\label{eq:endpoint_RN_derivative}
\frac{\mathrm d\pi}
     {\mathrm d\mathbb W_{01}^{\sigma}}
(R,S)
=
\frac{\rho_1(S)}
     {p_1^\sigma(R,S)}.
\end{equation}
Since $\rho_1$ is continuous and strictly positive on the compact
manifold $\mathsf G$, there exist constants
$0<m_1\leq M_1<\infty$ such that
\(
m_1\leq\rho_1(S)\leq M_1,
\;
 \text{for all}\; \,S\in\mathsf G.
\)
Combining this with \eqref{eq:heat_kernel_bounds}, the function
$\rho_1(S)/p_1^\sigma(R,S)$ is bounded above and away from zero on
$\mathsf G\times\mathsf G$. Consequently,
\(
D_{\mathrm{KL}}
\left(
\pi
\parallel
\mathbb W_{01}^{\sigma}
\right)
<\infty.
\)
Define a probability measure $\overline{\mathbb P}$ on $\mathcal X$ by
\begin{equation}
\label{eq:finite_entropy_competitor}
\frac{\mathrm d\overline{\mathbb P}}
     {\mathrm d\mathbb W_{\rho_0}^{\sigma}}
=
\frac{\rho_1(R_1)}
     {p_1^\sigma(R_0,R_1)}.
\end{equation}
Indeed, using the joint endpoint law
\eqref{eq:reference_endpoint_law},
\begin{align}
&\mathbb E_{\mathbb W_{\rho_0}^{\sigma}}
\left[
\frac{\rho_1(R_1)}
     {p_1^\sigma(R_0,R_1)}
\right] \notag \\
&=
\int_{\mathsf G\times\mathsf G}
\frac{\rho_1(S)}
     {p_1^\sigma(R,S)}
\rho_0(R)p_1^\sigma(R,S)
\,\mathrm d\mu(R)\,\mathrm d\mu(S)
\notag=
1,
\end{align}
so \eqref{eq:finite_entropy_competitor} defines a probability measure.
Moreover, the joint law of $(R_0,R_1)$ under
$\overline{\mathbb P}$ is precisely $\pi$. Therefore
\(
\overline{\mathbb P}\in\Pi_{01}.
\)
Furthermore,
\begin{align}
D_{\mathrm{KL}}
\left(
\overline{\mathbb P}
\parallel
\mathbb W_{\rho_0}^{\sigma}
\right)
&=
\mathbb E_{\overline{\mathbb P}}
\left[
\log
\frac{\rho_1(R_1)}
     {p_1^\sigma(R_0,R_1)}
\right]
\notag\\
&=
D_{\mathrm{KL}}
\left(
\pi
\parallel
\mathbb W_{01}^{\sigma}
\right)
<\infty.
\label{eq:finite_path_entropy_competitor}
\end{align}
Consequently,
\(
\inf_{\mathbb P\in\Pi_{01}}
D_{\mathrm{KL}}
\left(
\mathbb P
\parallel
\mathbb W_{\rho_0}^{\sigma}
\right)
<\infty.
\)
The relative entropy
\(
\mathbb P\mapsto
D_{\mathrm{KL}}
\left(
\mathbb P\parallel\mathbb W_{\rho_0}^{\sigma}
\right)
\)
is lower semicontinuous with respect to weak convergence of probability
measures, and its sublevel sets are weakly compact. Moreover,
$\Pi_{01}$ is weakly closed. Consequently, the
path space problem \eqref{eq:SBPKL} admits a minimizer
$\mathbb P^{\mathrm{opt}}\in\Pi_{01}$.

Furthermore, $\Pi_{01}$ is convex and
$D_{\mathrm{KL}}
\left(
\cdot
\parallel
\mathbb W_{\rho_0}^{\sigma}
\right)
$
is strictly convex on
$\mathcal M(\mathcal X)$.  Hence this minimizer is unique. Since every admissible control induces a measure in $\Pi_{01}$, \eqref{eq:SBPKL} is initially a relaxation of the original control problem. By the standard Schrödinger bridge characterization~\cite{leonard2013survey}, however, its minimizer is a Markov diffusion obtained through a Doob $h$-transform, whose drift is generated by a positive Schrödinger potential. Under the present assumptions and compactness of $\mathsf G$, this drift is bounded and defines an admissible finite-energy control. So the relaxation is exact, and the minimizer induces a minimizing pair for Problem~\ref{problem:SBP_on_G}.  
\end{proof} 

\begin{remark}
It is instructive to juxtapose the existence-uniqueness proof above with that in our prior work \cite[Theorem 1]{mah2026}. Beyond techniques, the nature of the arguments are different:~\cite[Theorem 1]{mah2026} only established the existence-uniqueness of solution \emph{a posteriori} since then, we had to proceed formally with necessary conditions for optimality, transform those conditions into new variables (Schr\"{o}dinger system), and then established existence-uniqueness for the resulting system. In contrast, the existence-uniqueness proof for the path space formulation established here is direct. 
\end{remark}

\subsection{Large-Deviation Interpretation}
\label{sec:large_deviation_interpretation}
The path space formulation also gives a large-deviation\footnote{The $D_{\mathrm{KL}}$ serves as the rate function.}
interpretation of the Schr\"odinger bridge problem. Indeed, consider
$N$ independent realizations
$R^{(1)},\ldots,R^{(N)}$
of the uncontrolled diffusion with common law
$\mathbb W_{\rho_0}^{\sigma}$, and define the empirical path measure
\(L_N
:=
\frac1N
\sum_{k=1}^N
\delta_{R^{(k)}}.
\)
By Sanov's large-deviation principle, the probability that $L_N$
is close to a path measure $\mathbb P$ behaves, at the exponential
scale
\(
\exp\!\left(
-N
D_{\mathrm{KL}}
\left(
\mathbb P
\parallel
\mathbb W_{\rho_0}^{\sigma}
\right)
\right).
\)
Therefore, among all path measures satisfying the prescribed endpoint
constraints, the minimizer
$\mathbb P^{\mathrm{opt}}$
is the least unlikely, or equivalently the most probable, deviation
from the uncontrolled dynamics. Since
$\mathbb P^{\mathrm{opt}}$
corresponds to
$\left(
\rho_\sigma^{\mathrm{opt}},
\Omega_\sigma^{\mathrm{opt}}
\right)$, the optimal Schr\"odinger bridge may thus be interpreted as the most
likely stochastic evolution connecting the two prescribed marginals.    
\section{Static Sinkhorn and Numerical Example} \label{sec:numerical_simulations}
\begin{figure*}[!t]
    \centering
\includegraphics[width=0.97\textwidth]{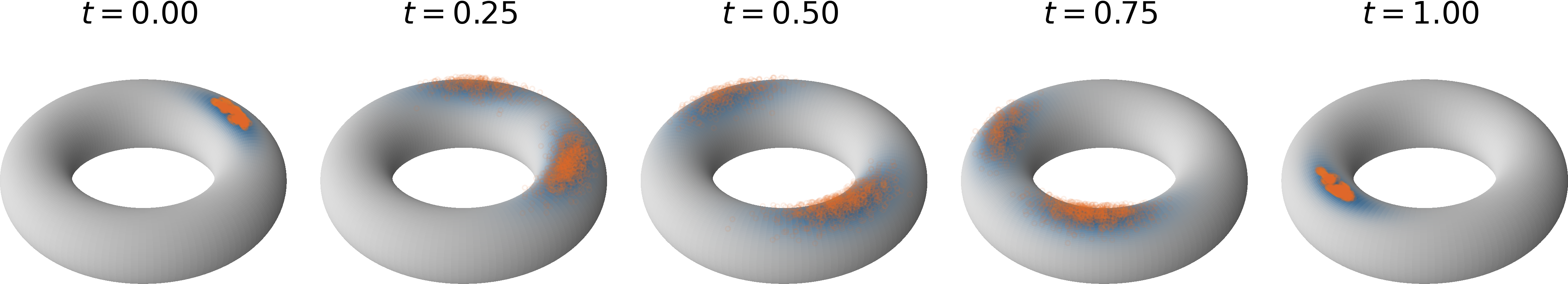}
\vspace{-6pt}
    \caption{Evolution of the Schrödinger bridge probability density between endpoint distributions  $\rho_0$ and $\rho_1$ on the torus
    $\mathbb{T}^{2}$. In {\blue{blue}} is shown the continuous density function $\rho^{\mathrm{opt}}_{\sigma}(\cdot,t)$ -- the solution of \eqref{eq:SBP_on_G} computed by solving the boundary-coupled system of heat PDEs on the Lie group as in \cite{mah2026}. In {\color{orange}{orange}} are the scattered data points transported along  $\mathbb P^{\mathrm{opt}}$ -- the minimizer of \eqref{eq:SBPKL} computed by solving the static Sinkhorn recursions.}
    \label{fig:torusSnapshots}
    \vspace*{-0.1in}
\end{figure*}
To practically demonstrate the equivalence between the solutions of \eqref{eq:SBP_on_G} and \eqref{eq:SBPKL}, we solve both formulations for the following 
antipodal unimodal reference measures $\rho_0$, $\rho_1$ supported on the torus $\mathbb{T}^{2}$:
\begin{align*}
    \rho_0(\theta_1,\theta_2) &= \frac{\exp\left(25\cos(\theta_1-\tfrac{\pi}{3}) + 30\cos(\theta_2-\tfrac{\pi}{2})\right)}{4\pi^2 I_0(25)\,I_0(30)},\\
    \rho_1(\theta_1,\theta_2) &= \frac{\exp\left(25\cos(\theta_1-\tfrac{11\pi}{8}) + 30\cos(\theta_2-\tfrac{\pi}{2})\right)}{4\pi^2 I_0(25)\,I_0(30)},
\end{align*}
where the Bessel function $I_0(\kappa) := \sum_{m=0}^{\infty} \frac{1}{(m!)^2}\left(\frac{\kappa}{2}\right)^{2m}$.

To solve \eqref{eq:SBP_on_G}, the domain $\mathbb{T}^2$ for both measures is discretized into a $128\times 128$ grid, whereover a Schr\"odinger system (boundary coupled system of forward and backward heat PDEs over the Lie group) is solved by \emph{dynamic} Sinkhorn recursions; see \cite{mah2026}.
For solving \eqref{eq:SBPKL}, we make use of strict convexity of the objective thereof, proven in Theorem \ref{thm:existence-uniqueness_SBPsolution}. The resulting strong duality allows for the solution of \eqref{eq:SBPKL} via \emph{static} Sinkhorn recursions over the Lagrange multipliers associated with the constraints of $\Pi_{01}$. For the static Sinkhorn recursions, the guarantee of linear convergence to a unique set of multipliers is well-established; see e.g., \cite[Sec. V]{bondar2025stochastic}. Following convergence, the minimizer $\mathbb{P}^{\rm{opt}}$ can be obtained from the multipliers. In our implementation, we sample $N=50$ points from the discretized $\rho_0$ and $\rho_1$ and solve \eqref{eq:SBPKL} therebetween\footnote{The choice to solve \eqref{eq:SBPKL} over scattered data is for numerical convenience of static Sinkhorn; one can use the same fixed grid as for the solution of~\eqref{eq:SBP_on_G}.}.
In Figure
\ref{fig:torusSnapshots} we see the SB between the two measures is the same when
solving both formulations. The 50 sampled data points used in the solution of
\eqref{eq:SBPKL} are shown to track exactly the continuous density function, which is the solution of \eqref{eq:SBP_on_G}. More simulation and animations are available~at:\\ {\small{\url{https://gradslab.github.io/LargeDeviationSBP/}}}

\section{Concluding Remarks}
\label{sec:Conclusion} 
We established equivalence between the stochastic optimal control and path space formulations of the SBP on a compact connected Lie group. Using horizontal lift and stochastic anti-development, we derived a Girsanov type change of measure, and related the expected control energy to path space relative entropy. This equivalence yields existence and uniqueness of the bridge, a large-deviation interpretation, and a static Sinkhorn formulation for computing the optimal path measure. Numerical results on $\mathbb{T}^2$ illustrate agreement between the stochastic optimal control and path space solutions.
 
\bibliographystyle{IEEEtran}
\bibliography{myreferences}

\end{document}